\documentclass[final]{amsart}
\usepackage{amssymb}
\usepackage{array}
\usepackage{amsthm}
\usepackage[all]{xy}

\newtheorem{theorem}{Theorem}

\newtheorem{proposition}[theorem]{Proposition}
\newtheorem{corollary}[theorem]{Corollary}

\theoremstyle{definition}

\newtheorem{question}[theorem]{Question}

\newcommand{\IN}{\mathbb N}

\newcommand{\IZ}{\mathbb Z}
\newcommand{\IC}{\mathbb C}

\title{Difference bases in dihedral groups}
\author{Taras Banakh and Volodymyr Gavrylkiv}

\address[T.~Banakh]{Ivan Franko National University of Lviv (Ukraine),
and \newline Institute of Mathematics, Jan Kochanowski University in
Kielce (Poland)}
\email{t.o.banakh@gmail.com}
\address[V.~Gavrylkiv]{Vasyl Stefanyk Precarpathian National
University,
Ivano-Frankivsk, Ukraine} \email{vgavrylkiv@gmail.com}

\begin{document}

\begin{abstract} A subset $B$ of a group $G$ is called a {\em
difference basis} of $G$ if each element $g\in G$ can be written as the
difference $g=ab^{-1}$ of some elements $a,b\in B$. The smallest
cardinality $|B|$ of a difference basis $B\subset G$ is called the {\em
difference size} of $G$ and is denoted by $\Delta[G]$.
The fraction
$\eth[G]:=\Delta[G]/{\sqrt{|G|}}$ is called the {\em difference characteristic} of $G$.
We prove that for every $n\in\IN$ the dihedral group
$D_{2n}$ of order $2n$ has the difference characteristic
$\sqrt{2}\le\eth[D_{2n}]\leq\frac{48}{\sqrt{586}}\approx1.983$.
Moreover, if $n\ge 2\cdot 10^{15}$, then  $\eth[D_{2n}]<\frac{4}{\sqrt{6}}\approx1.633$.
Also we calculate the difference sizes and characteristics of all dihedral groups of cardinality $\le80$.
\end{abstract}
\maketitle



A subset $B$ of a group $G$ is called a {\em difference basis} for a subset $A\subset G$ if each element $a\in A$
can be written as $a=xy^{-1}$ for some $x,y\in B$.
 The smallest cardinality of a difference basis for $A$ is called the {\em difference size} of $A$ and is denoted by $\Delta[A]$. For example, the set $\{0,1,4,6\}$ is a difference basis for the interval $A=[-6,6]\cap\IZ$ witnessing that $\Delta[A]\le 4$.

The definition of a difference basis $B$ for a set $A$ in a group $G$ implies that $|A|\le |B|^2$ and gives a lower bound $ \sqrt{|A|}\le\Delta[A]$.
The fraction
$$\eth[A]:=\frac{\Delta[A]}{\sqrt{|A|}}\ge1$$is called the {\em difference characteristic} of $A$.

For a real number
 $x$ we put $$\lceil x\rceil=\min\{n\in\IZ:n\ge x\}\mbox{ and }\lfloor
x\rfloor=\max\{n\in\IZ:n\le x\}.$$

The following proposition is proved in \cite[1.1]{BGN}.

\begin{proposition}\label{p:BGN} Let $G$
be a finite group. Then
\begin{enumerate}
\item $ \frac{1+\sqrt{4|G|-3}}2\le \Delta[G]\le
\big\lceil\frac{|G|+1}2\big\rceil$,
\item $\Delta[G]\le \Delta[H]\cdot \Delta[G/H]$ and
$\eth[G]\le\eth[H]\cdot\eth[G/H]$ for any normal subgroup
$H\subset G$;
\item $\Delta[G]\le |H|+|G/H|-1$ for any subgroup $H\subset G$.
\end{enumerate}
\end{proposition}

In \cite{KL} Kozma and Lev
proved (using the classification of finite simple groups) that each finite group $G$ has difference
characteristic $\eth[G]\le\frac{4}{\sqrt{3}}\approx 2.3094$.

In this paper we shall evaluate the difference characteristics of dihedral groups and prove that each diherdal group $D_{2n}$ has $\eth[D_{2n}]\leq\frac{48}{\sqrt{586}}\approx 1.983$. Moreover, if $n\ge 2\cdot 10^{15}$, then  $\eth[D_{2n}]<\frac{4}{\sqrt{6}}\approx1.633$. We recall that the {\em dihedral group} $D_{2n}$ is the isometry group of a regular $n$-gon. The dihedral group $D_{2n}$ contains a normal cyclic subgroup of index 2. A standard model of a cyclic group of order $n$ is the multiplicative
group $$C_n=\{z\in\IC:z^n=1\}$$ of $n$-th roots of $1$.
The group $C_n$ is isomorphic to the additive group of the ring $\IZ_n=\IZ/n\IZ$.





\begin{theorem}\label{t:dihedral} For any numbers $n,m\in\IN$ the dihedral group
$D_{2nm}$ has the difference size $$2\sqrt{nm}\le \Delta[D_{2nm}]\le
\Delta[D_{2n}]\cdot\Delta[C_m]$$and the difference characteristic
$\sqrt{2}\le\eth[D_{2nm}]\le \eth[D_{2n}]\cdot\eth[C_{m}]$.
\end{theorem}

\begin{proof} It is well-known that the dihedral group $D_{2nm}$ contains
a normal cyclic subgroup of order $nm$, which can be identified with the cyclic group $C_{nm}$. The subgroup $C_m\subset C_{nm}$ is normal in $D_{2mn}$ and the
quotient group $D_{2mn}/C_m$ is isomorphic to $D_{2n}$. Applying
Proposition~\ref{p:BGN}(2), we obtain the upper bounds $\Delta[D_{2n}]\le
\Delta[D_{2nm}/C_m]\cdot\Delta[C_m]=\Delta[D_{2n}]\cdot\Delta[C_m]$
and $\eth[D_{2nm}]\le\eth[D_{2n}]\cdot\eth[C_{m}]$.
\smallskip

Next, we prove the lower bound  $2\sqrt{nm}\le\Delta[D_{2nm}]$. Fix any
element $s\in D_{2nm}\setminus C_{nm}$ and observe that $s=s^{-1}$ and
$sxs^{-1}=x^{-1}$ for all $x\in C_{nm}$.
Fix a difference basis $D\subset D_{2nm}$ of cardinality
$|D|=\Delta[D_{2nm}]$ and write $D$ as the union $D=A\cup sB$ for some
sets $A,B\subset C_{nm}\subset D_{2nm}$. We claim that $AB^{-1}=C_{nm}$. Indeed,
for any $x\in C_{nm}$ we get $xs\in sC_{nm}\cap (A\cup sB)(A\cup
sB)^{-1}=AB^{-1}s^{-1}\cup sBA^{-1}$ and hence $$x\in
AB^{-1}s^{-1}s^{-1}\cup sBA^{-1}s^{-1}=AB^{-1}\cup B^{-1}A=AB^{-1}.$$
So, $C_{nm}=AB^{-1}$ and hence $nm\le|A|\cdot |B|$. Then
$\Delta[D_{2nm}]=|A|+|B|\ge \min\{l+k:l,k\in\IN,\;lk\ge nm\}\ge2\sqrt{nm}$
and
$\eth[D_{2nm}]=\frac{\Delta[D_{2nm}]}{\sqrt{2nm}}\ge\frac{2\sqrt{nm}}{\sqrt{2nm}}=\sqrt{2}$.
\end{proof}

\begin{corollary}\label{c:dihedral} For any number $n\in\IN$ the dihedral group
$D_{2n}$ has the difference size $$2\sqrt{n}\le \Delta[D_{2n}]\le
2\cdot\Delta[C_n]$$and the difference characteristic
$\sqrt{2}\le\eth[D_{2n}]\le \sqrt{2}\cdot\eth[C_{n}]$.
\end{corollary}

The difference sizes of finite cyclic groups were evaluated in \cite{BG} with the help of the  difference sizes
of the order-intervals $[1,n]\cap\IZ$ in the additive group $\IZ$ of integer numbers.
For a natural number $n\in\IN$ by $\Delta[n]$ we shall denote the difference size of
the order-interval $[1,n]\cap\IZ$ and by $\eth[n]:=\frac{\Delta[n]}{\sqrt{n}}$ its difference characteristic.
The asymptotics of the sequence $(\eth[n])_{n=1}^\infty$ was studied by R\'edei and R\'enyi \cite{RR}, Leech \cite{Leech}
and Golay \cite{Golay} who eventually proved that $$\sqrt{2+\tfrac4{3\pi}}<
\sqrt{2+\max_{0<\varphi<2\pi}\tfrac{2\sin(\varphi)}{\varphi+\pi}}\le \lim_{n\to\infty}\eth[n]=\inf_{n\in\IN}\eth[n]\le
\eth[6166]=\frac{128}{\sqrt{6166}}<\eth[6]=\sqrt{\tfrac{8}3}.$$

In \cite{BG} the difference sizes of the order-intervals $[1,n]\cap\IZ$ were applied to give
upper bounds for the difference sizes of finite cyclic groups.

\begin{proposition}\label{p:c<n} For every $n\in\IN$ the cyclic group $C_n$ has difference size $\Delta[C_n]\le\Delta\big[\lceil\frac{n-1}2\rceil\big]$,
which implies that
$$\limsup_{n\to\infty}\eth[C_n]\le\frac1{\sqrt{2}}\inf_{n\in\IN}\eth[n]\le\frac{64}{\sqrt{3083}}<\frac{2}{\sqrt{3}}.$$
\end{proposition}

The following upper bound for the difference sizes of cyclic groups were proved in \cite{BG}.

\begin{theorem}\label{t:cyclic} For any $n\in\IN$ the cyclic group $C_n$ has the difference characteristic:
\begin{enumerate}
\item $\eth[C_n]\le\eth[C_4]=\frac32$;
\item $\eth[C_n]\le\eth[C_2]=\eth[C_8]=\sqrt{2}$ if $n\ne 4$;
\item $\eth[C_n]\le\frac{12}{\sqrt{73}}<\sqrt{2}$ if $n\ge 9$;
\item $\eth[C_n]\le\frac{24}{\sqrt{293}}<\frac{12}{\sqrt{73}}$ if $n\ge 9$ and $n\ne 292$;
\item $\eth[C_n]<\frac2{\sqrt{3}}$ if $n\ge 2\cdot 10^{15}$.
\end{enumerate}
\end{theorem}

For some special numbers $n$ we have more precise upper bounds for $\Delta[C_n]$. A number $q$ is called
a {\em prime power} if $q=p^k$ for some prime number $p$ and some $k\in\IN$.

The following theorem was derived in \cite{BG} from the classical results of Singer \cite{Singer},
Bose, Chowla \cite{Bose}, \cite{Chowla} and Rusza \cite{Rusza}.

\begin{theorem}\label{Singer} Let $p$ be a prime number and $q$ be a prime power.
Then
\begin{enumerate}
\item $\Delta[C_{q^2+q+1}]=q+1$;
\item $\Delta[C_{q^2-1}]\le q-1+\Delta[C_{q-1}]\le q-1+\frac{3}2\sqrt{q-1}$;
\item $\Delta[C_{p^2-p}]\le p-3+\Delta[C_{p}]+\Delta[C_{p-1}]\le
p-3+\frac32(\sqrt{p}+\sqrt{p-1})$.
\end{enumerate}
\end{theorem}

The following Table~\ref{tab:cycl} of difference sizes and characteristics of cyclic groups $C_n$ for $\le 100$ is taken from \cite{BG}.

\begin{table}[ht]
\caption{Difference sizes and characteristics of cyclic groups $C_n$ for $n\le100$}\label{tab:cycl}
\begin{tabular}{|c|c|c||c|c|c||c|c|c||c|c|c|}
\hline
$n$   & \!$\Delta[C_n]$\! & $\eth[C_n]$&$n$   & \!$\Delta[C_n]$\! & $\eth[C_n]$&$n$   & \!$\Delta[C_n]$\! & $\eth[C_n]$&$n$   & \!$\Delta[C_n]$\! & $\eth[C_n]$\\
\hline
1 & 1 & 1 	&26 & 6 & 1.1766...\!\! & 51 & 8 & 1.1202...\!\!& 76 & 10 & 1.1470...\\
2 & 2 &1.4142...  &27 & 6 & 1.1547...\!\!      & 52 & 9 & 1.2480...\!\!& 77 & 10 & 1.1396...\!\!\\
3 & 2 &1.1547...  &28 & 6 & 1.1338...\!\! 	& 53 & 9 & 1.2362...\!\!& 78 & 10 & 1.1322...\!\!\\
4 & 3 &1.5  	  &29 & 7 & 1.2998...\!\! 	& 54 & 9 & 1.2247...\!\!& 79 & 10 & 1.1250...\!\!\\
5 & 3 &1.3416...  &30 & 7 & 1.2780...\!\! 	& 55 & 9 & 1.2135...\!\!& 80 & 11 & 1.2298...\!\!\\
6 & 3 & 1.2247... &31 & 6 & 1.0776...\!\! & 56 & 9 & 1.2026...\!\!& 81 & 11 & 1.2222...\!\!\\
7 & 3 & 1.1338... &32 & 7 & 1.2374...\!\! & 57 & 8 & 1.0596...\!\!& 82 & 11 & 1.2147...\!\!\\
8 & 4 & 1.4142... &33 & 7 & 1.2185...\!\! & 58 & 9 & 1.1817...\!\!& 83 & 11 & 1.2074...\!\!\\
9 & 4 & 1.3333... 	  &34 & 7 & 1.2004...\!\! & 59 & 9 & 1.1717...\!\!& 84 & 11 & 1.2001...\!\!\\
10 & 4 & 1.2649... &35 & 7 & 1.1832...\!\! & 60 & 9 & 1.1618...\!\!& 85 & 11 & 1.1931...\!\!\\
11 & 4 & 1.2060... &36 & 7 & 1.1666...\!\! 	& 61 & 9 & 1.1523...\!\!& 86 & 11 & 1.1861...\!\!\\
12 & 4 & 1.1547... &37 & 7 & 1.1507...\!\! 	& 62 & 9 & 1.1430...\!\!& 87 & 11 & 1.1793...\!\!\\
13 & 4 & 1.1094... &38 & 8 & 1.2977...\!\! 	& 63 & 9 & 1.1338...\!\!& 88 & 11 & 1.1726...\!\!\\
14 & 5 & 1.3363... &39 & 7 & 1.1208...\!\! 	& 64 & 9 & 1.125\!\!& 89 & 11 & 1.1659...\!\!\\
15 & 5 & 1.2909... &40 & 8 & 1.2649...\!\! 	& 65 & 9 & 1.1163...\!\!& 90 & 11 & 1.1595...\!\!\\
16 & 5 & 1.25     &41 & 8 & 1.2493...\!\! 	& 66 & 10 & 1.2309...\!\!& 91 & 10 & 1.0482...\!\!\\
17 & 5 & 1.2126... &42 & 8 & 1.2344...\!\! 	& 67 & 10 & 1.2216...\!\!& 92 & 11 & 1.1468...\!\!\\
18 & 5 & 1.1785... &43 & 8 & 1.2199...\!\! 	& 68 & 10 & 1.2126...\!\!& 93 & 12 & 1.2443...\!\!\\
19 & 5 & 1.1470... &44 & 8 & 1.2060...\!\! 	& 69 & 10 & 1.2038...\!\!& 94 & 12 & 1.2377...\!\!\\
20 & 6 & 1.3416... &45 & 8 & 1.1925...\!\! 	& 70 & 10 & 1.1952...\!\!& 95 & 12 & 1.2311...\!\!\\
21 & 5 & 1.0910... &46 & 8 & 1.1795...\!\! 	& 71 & 10 & 1.1867...\!\!& 96 & 12 & 1.2247...\!\!\\
22 & 6 & 1.2792... &47 & 8 & 1.1669...\!\! 	& 72 & 10 & 1.1785...\!\!& 97 & 12 & 1.2184...\!\!\\
23 & 6 & 1.2510... &48 & 8 & 1.1547...\!\! 	& 73 & 9 & 1.0533...\!\!& 98 & 12 & 1.2121...\!\!\\
24 & 6 & 1.2247... &49 & 8 & 1.1428...\!\! 	& 74 & 10 & 1.1624...\!\!& 99 & 12 & 1.2060...\!\!\\
25 & 6 & 1.2      &50 & 8 & 1.1313...\!\! 	& 75 & 10 & 1.1547...\!\!& 100 & 12 & 1.2\\
\hline
\end{tabular}
\end{table}

Using Theorem~\ref{Singer}(1), we shall prove that for infinitely many numbers $n$ the lower and upper bounds given in
Theorem~\ref{t:dihedral} uniquely determine the difference size
$\Delta[D_{2n}]$ of $D_{2n}$.

\begin{theorem}\label{t:dup} If $n=1+q+q^2$ for some prime power $q$,
then
$$\Delta[D_{2n}]=2\cdot \Delta[C_n]=\left\lceil
2\sqrt{n}\,\right\rceil=\Big\lceil\sqrt{2|D_{2n}|}\,\Big\rceil=2+2q.$$
\end{theorem}

\begin{proof} By Theorem~\ref{Singer}(1), $\Delta[C_n]=1+q$. Since
$$2\sqrt{q^2+q+1}=2\sqrt{n}\le \Delta[D_{2n}]\le \Delta[D_2]\cdot \Delta[C_n]=2\cdot \Delta[C_n]=2+2q,$$ it suffices to
check that $(2+2q)-2\sqrt{q^2+q+1}<1$, which is equivalent to
$\sqrt{q^2+q+1}>q+\frac12$ and to $q^2+q+1>q^2+q+\frac14$.
\end{proof}

A bit weaker result holds also for the dihedral groups $D_{8(q^2+q+1)}$.

\begin{proposition}\label{t:8dup} If $n=1+q+q^2$ for some prime power $q$,
then
$$4q+3\le\Delta[D_{8n}]\le4q+4.$$
\end{proposition}

\begin{proof} By Theorem~\ref{Singer}(1), $\Delta[C_n]=1+q$. Since $\Delta[D_8]=4$ (see Table~\ref{tab:d}), by Theorem~\ref{t:dihedral},
$$4\sqrt{q^2+q+1}=2\sqrt{4n}\le \Delta[D_{8n}]\le \Delta[D_{8}]\cdot \Delta[C_n]=4(1+q).$$ To see that $4q+3\le \Delta[D_{8n}]\le 4q+4$, it suffices to check that  $(4+4q)-4\sqrt{q^2+q+1}<2$, which is equivalent to
$\sqrt{q^2+q+1}>q+\frac12$ and to $q^2+q+1>q^2+q+\frac14$.
\end{proof}

In Table~\ref{tab:d} we present the results of computer calculation of
the difference sizes and characteristics of dihedral groups of order $\le 80$. In this
table  $lb[D_{2n}]:=\lceil \sqrt{4n}\,\rceil$ is the lower bound given
in Theorem~\ref{t:dihedral}.
With the boldface font we denote the numbers $2n\in\{14,26,42,62\}$,
equal to $2(q^2+q+1)$ for a prime power $q$. For these numbers we know that
$\Delta[D_{2n}]=lb[D_{2n}]=2q+2$. For $q=2$ and $n=q^2+q+1=7$ the table shows that $\Delta[D_{56}]=\Delta[D_{8n}]=11=4q+3$, which means that the lower bound $4q+3$ in Proposition~\ref{t:8dup} is attained.

\begin{table}[ht]
\caption{Difference sizes and characteristics of dihedral groups $D_{2n}$ for
$2n\le80$.}\label{tab:d}
\begin{tabular}{|c|c|c|c|c||c|c|c|c|c|}
\hline
$2n$   & \!$lb[D_{2n}]$\! & $\Delta[D_{2n}]$ &$2\Delta[C_n]$   & \!$\eth[D_{2n}]$ & $2n$   & \!$lb[D_{2n}]$\! & $\Delta[D_{2n}]$ &$2\Delta[C_n]$   & \!$\eth[D_{2n}]$\\
\hline
2 & 2 & 2 & 2 & 1.4142...\!\! & {\bf 42} & 10 & 10 & 10 & 1.5430...\!\!\\
4 & 3 & 3 & 4 & 1.5\!\! & 44 & 10 & 10 & 12 & 1.5075...\!\!\\
6 & 4 & 4 & 4 & 1.6329...\!\! & 46 & 10 & 11 & 12 & 1.6218...\!\!\\
8 & 4 & 4 & 6 & 1.4142...\!\! & 48 & 10 & 10 & 12 & 1.4433...\!\!\\
10 & 5 & 5 & 6 & 1.5811...\!\! & 50 & 10 & 11 & 12 & 1.5556...\!\!\\
12 & 5 & 5 & 6 & 1.4433...\!\! & 52 & 11 & 11 & 12 & 1.5254...\!\!\\
{\bf 14} & 6 & 6 & 6 & 1.6035...\!\! & 54 & 11 & 12 & 12 & 1.6329...\!\!\\
16 & 6 & 6 & 8 & 1.5\!\! & 56 & 11 & 11 & 12 & 1.4699...\!\!\\
18 & 6 & 7 & 8 & 1.6499...\!\! & 58 & 11 & 12 & 14 & 1.5756...\!\!\\
20 & 7 & 7 & 8 & 1.5652...\!\! & 60 & 11 & 12 & 14 & 1.5491...\!\!\\
22 & 7 & 8 & 8 & 1.7056...\!\! & {\bf 62} & 12 & 12 & 12 & 1.5240...\!\!\\
24 & 7 & 7 & 8 & 1.4288...\!\! & 64 & 12 & 12 & 14 & 1.5\!\!\\
{\bf 26} & 8 & 8 & 8 & 1.5689...\!\! & 66 & 12 & 13 & 14 & 1.6001...\!\!\\
28 & 8 & 8 & 10 & 1.5118...\!\! & 68 & 12 & 13 & 14 & 1.5764...\!\!\\
30 & 8 & 8 & 10 & 1.4605...\!\! & 70 & 12 & 12 & 14 & 1.4342...\!\!\\
32 & 8 & 9 & 10 & 1.5909...\!\! & 72 & 12 & 13 & 14 & 1.5320...\!\!\\
34 & 9 & 9 & 10 & 1.5434...\!\! & 74 & 13 & 14 & 14 & 1.6274...\!\!\\
36 & 9 & 9 & 10 & 1.5\!\! & 76 & 13 & 14 & 16 & 1.6059...\!\!\\
38 & 9 & 10 & 10 & 1.6222...\!\! & 78 & 13 & 14 & 14 & 1.5851...\!\!\\
40 & 9 & 9 & 12 & 1.4230...\!\! & 80 & 13 & 14 & 16 & 1.5652...\!\!\\
\hline
\end{tabular}
\end{table}

\begin{theorem}\label{t:max} For any number $n\in \IN$ the dihedral group
$D_{2n}$ has the difference characteristic
$$\sqrt{2}\le\eth[D_{2n}]\leq\frac{48}{\sqrt{586}}\approx 1.983.$$
Moreover, if $n\ge 2\cdot 10^{15}$, then  $\eth[D_{2n}]<\frac{4}{\sqrt{6}}\approx1.633 $.
\end{theorem}

\begin{proof} By Corollary~\ref{c:dihedral},  $\sqrt{2}\le\eth[D_{2n}]\le \sqrt{2}\cdot\eth[C_{n}]$.
If $n\ge 9$ and $n\ne 292$, then $\eth[C_n]\le\frac{24}{\sqrt{293}}$ by Theorem~\ref{t:cyclic}(4), and
hence
$\eth[D_{2n}]\le \sqrt{2}\cdot\eth[C_{n}]\leq\sqrt{2}\cdot\frac{24}{\sqrt{293}}=\frac{48}{\sqrt{586}}$.
If $n=292$, then known values $\eth[C_{73}]=\frac{9}{\sqrt{73}}$ (given in Table~\ref{tab:cycl}),
$\eth[D_{8}]=\frac{4}{\sqrt{8}}=\sqrt{2}$ (given in Table~\ref{tab:d}) and Theorem~\ref{t:dihedral} yield the upper bound
$$\eth[D_{2\cdot 292}]=\eth[D_{8\cdot 73}]\le\eth[D_8]\cdot\eth[C_{73}]=\sqrt{2}\cdot \frac{9}{\sqrt{73}}<\frac{48}{\sqrt{586}}.$$

Analyzing the data from Table~\ref{tab:d}, one can check that
$\eth[D_{2n}]\le \frac{48}{\sqrt{586}}\approx 1.983$ for all $n\le 8$.

If $n\ge 2\cdot 10^{15}$, then $\eth[C_{n}]<\frac{2}{\sqrt{3}}$ by Theorem~\ref{t:cyclic}(5), and hence
$$\eth[D_{2n}]\leq\sqrt{2}\cdot\eth[C_{n}]<\frac{4}{\sqrt{6}}.$$

\end{proof}

\begin{question}\label{q:822} Is
$\sup_{n\in\IN}\eth[D_{2n}]=\eth[D_{22}]=\frac{8}{\sqrt{22}}\approx1.7056$?
\end{question}

To answer Question~\ref{q:822} affirmatively, it suffices to check that $\eth[D_{2n}]\le\frac8{\sqrt{22}}$ for all $n<1\,212\,464$.

\begin{proposition} The inequality $\eth[D_{2n}]\le\sqrt{2}\cdot\eth[C_n]\le \frac8{\sqrt{22}}$ holds for all $n\ge 1\,212\,464$.
\end{proposition}

\begin{proof} It suffices to prove that $\eth[C_n]\le\frac4{\sqrt{11}}$ for all $n\ge 1\,212\,464$.
To derive a contradiction, assume that $\eth[C_n]>\frac4{\sqrt{11}}$ for some $n\ge 1\,212\,464$.
Let $(q_k)_{k=1}^\infty$ be an increasing enumeration of prime powers.
Let $k\in\IN$ be the unique number such that $12q_{k}^2+14q_{k}+15<n\le 12q_{k+1}^2+14q_{k+1}+15$.
By Corollary 4.9 of \cite{BG}, $\Delta[C_n]\le 4(q_{k+1}+1)$. The inequality $\eth[C_n]>\frac4{\sqrt{11}}$ implies $$4(q_{k+1}+1)\ge \Delta[C_n]>\frac4{\sqrt{11}}\sqrt{n}\ge \frac4{\sqrt{11}}\sqrt{12q_{k}^2+14q_{k}+16}.$$ By Theorem 1.9 of \cite{Dusart}, if $q_{k}\ge 3275$, then $q_{k+1}\le q_{k}+\frac{q_{k}}{2\ln^2(q_{k})}$. On the other hand, using WolframAlpha computational knowledge engine it can be shown that the inequality $1+x+\frac{x}{2\ln^2(x)}\le \frac1{\sqrt{11}}\sqrt{12x^2+14x+16}$ holds for all $x\ge 43$. This implies that $q_{k}<3275$.

Analysing the table\footnote{See {\tt https://primes.utm.edu/notes/GapsTable.html}  and {\tt https://primes.utm.edu/lists/small/1000.txt}} of (maximal gaps between) primes, it can be shows that $11(q_{k+1}+1)^2\le {12q_k^2+14q_k+16}$ if $q_{k}\ge 331$.  So, $q_k\le 317$, $q_{k+1}\le 331$ and
$11\cdot (q_{k+1}+1)^2=11\cdot 332^2=1\,212\,464\le n$, which contradicts $4(q_{k+1}+1)>\frac4{\sqrt{11}}\sqrt{n}$.
\end{proof}

\end{document}